\documentclass{amsart}

\usepackage{amsmath}
\usepackage{amssymb}
\usepackage{amscd}
\usepackage{hyperref}

\newtheorem{theorem}{Theorem}
\newtheorem{thm}[theorem]{Theorem}
\newtheorem{cor}[theorem]{Corollary}

\theoremstyle{definition}
\newtheorem{defn}[theorem]{Definition}

\newtheorem{ques}[theorem]{Question}

\newtheorem{rem}[theorem]{Remark}

\theoremstyle{remark}

\newcommand{\mbb}{\mathbb}
\newcommand{\QQ}{\mbb{Q}}

\newcommand{\ZZ}{\mbb{Z}}

\newcommand{\PP}{\mbb{P}}

\newcommand{\mc}{\mathcal}

\newcommand{\mcX}{\mc{X}}

\newcommand{\fF}{\kappa}

\newcommand{\mfm}{\mathfrak{m}}

\newcommand{\OO}{\mc{O}}

\newcommand{\ol}{\overline}

\newcommand{\SP}{\text{Spec }}

\newsavebox{\sembox}
\newlength{\semwidth}
\newlength{\boxwidth}

\newsavebox{\semrbox}
\newlength{\semrwidth}
\newlength{\boxrwidth}

\title
{Separable rational connectedness and weak approximation in positive characteristic}

\author[Starr]{Jason Michael Starr}
\address{Department of Mathematics \\
  Stony Brook University \\ Stony Brook, NY 11794 USA}
\email{jstarr@math.sunysb.edu} 

\author[Tian]{Zhiyu Tian}
\address{
Beijing International Center for Mathematical Research \\
Peking University \\
No.5 Yiheyuan Road \\
Haidian District, Beijing \\ 
China}
\email{zhiyutian@bicmr.pku.edu.cn}

\date{\today}

\begin{document}


\begin{abstract}
In this short note we give a characterization of smooth projective varieties of Picard number one that are separably uniruled but not separably rationally connected. We also give a sufficient condition involving the torsion order and the uniruling index for a smooth Fano variety of Picard number one to be separably rationally connected. As an application, we prove some weak approximation results for Fano complete intersections in positive charactersitic. For example, we show that weak approximation holds at place of strong potentially good reduction for a Fano complete intersection in $\PP^n$ of type $(d_1, \ldots, d_c)$ in characteristic $p$ such that $n>d_1+\ldots d_c, p>d_1, \ldots, d_c.$
\end{abstract}


\maketitle


In this short note, we improve some results in \cite{hypersurface} and \cite{WASTZ2018} about separable rational connectedness and weak approximation for Fano complete intersections in positive characteristic.

First let us recall some basic definitions.

\begin{defn}
For every field $K$, for every integral, $n$-dimensional,
  projective $K$-scheme $X$,
  for every $r>0$, an $r$-\textbf{uniruling} of $X$ over $K$,
  $$
  (h,\pi):Y \to X \times_{\SP K} M,
  $$
  is a finite morphism of $K$-varieties
  such that $\pi$ is proper and smooth with geometric fibers $\PP^1$
  and such that the $K$-morphism from the
  $r$-fold fiber product is dominant,
  $$
  h^{(r)}:Y\times_M \dots \times_M Y \to X \times_{\SP K}\dots
  \times_{\SP K} X, \ \ \text{pr}_i\circ  
  h^{(r)} = h\circ \text{pr}_i.
  $$  
  
  Following \cite[Definition IV.1.7.3]{Kollar96}, the $r$-\textbf{uniruling
    index}, $u_r(K,X)$, is the greatest common divisor of
  $\text{deg}(h^{(r)})$ for all $r$-unirulings with $h^{(r)}$
  generically finite ($0$ if there are no $r$-unirulings).  
  
  We say that $X_{\ol{K}}$ is
  \emph{uniruled }(resp. \emph{rationally connected}) if $u_1(K,X)>0$ (resp.  $u_2(K,X)>0$).
  
  We say $X_{\ol{K}}$ is \emph{separably uniruled} if there is a uniruling family such that $pr_1 \circ h^{(1)}: Y \to X$ is dominant and separable.

Similarly, $X_{\ol{K}}$ is \emph{separably rationally connected} if there is a $2$-uniruling family such that
$$
  h^{(2)}:Y\times_M  Y \to X \times_{\SP K}X, \ \ \text{pr}_i\circ  
  h^{(2)} = h\circ \text{pr}_i.
  $$  
 is dominant and separable.
\end{defn}

 The following is a well-known open question:
\begin{ques}\label{2}
Is every smooth Fano hypersurface separably rationally connected?
\end{ques}
Some partial answers are given by Chen-Zhu\cite{ChenZhuCI}, the work of the second named author \cite{hypersurface}, and in a recent joint work of the two authors with R. Zong \cite{WASTZ2018}. One key component of the joint work of authors with R. Zong is a result about closedness of the separable rational connectedness under some conditions, which we will generalize in this article.
 
We first strengthen a result of the second named author \cite[Theorem 5]{hypersurface}.
\begin{thm}\label{thm:PicOne}
Let $X$ be a smooth projective variety, whose Picard group is isomorphic to $\ZZ$, defined over an algebraically closed field of positive characteristic. Assume that $X$ is separably uniruled. Then $X$ is separably rationally connected if and only if $H^0(X, \Omega^i_X)=0$ for $i=1, \ldots, \dim X$. 
\end{thm}

\begin{rem}
One can prove that such $X$ in Theorem \ref{thm:PicOne} is freely rationally connected (FRC) as defined in \cite[Definition 1.2]{ShenFRC}.
\end{rem}

\begin{proof}[Proof of Theorem \ref{thm:PicOne}]
The ``only if" part is well-known. In the following we prove the ``if" part.

Since $X$ is separably uniruled, there is a free curve, i.e. a morphism $f:\PP^1 \to X$ such that 
\begin{align*}
f^{*}\Omega_X &\cong \OO(-a_1)\oplus \ldots \oplus \OO(-a_r) \oplus \OO \oplus \ldots \oplus \OO,\\
f^{*} T_X &\cong \OO(a_1)\oplus \ldots \oplus \OO(a_r) \oplus \OO \oplus \ldots \oplus \OO, a_1 \geq a_2 \geq \ldots a_r>0.\\
\end{align*}

Define the positive rank $r$ of $X$ to be the maximum number of non-trivial summand in the above decomposition among all free curves. A free curve is called \emph{maximally free} if the pull-back of the cotangent bundle has $r$ negative summands.

Given a general point $x \in X$, by \cite[Proposition 2.2]{ShenFRC}, there is a well-defined subspace $D(x) \subset \Omega_X|_x$, as the subspace of the $\OO$-directions of a maximally free curve at $x$ (i.e. $D(x)$ is independent of the choice of the maximally free curve). Furthermore,  over an open subset $U$ of $X$, which contains a general maximally free rational curve, the subspaces of $D(x)$ of $\Omega_X|_x$ glue together to a (locally free) coherent subsheaf of $\Omega_X$ (loc. cit. Proposition 2.5). Denote by $\mathcal{D}$ the saturated subsheaf of $\Omega_X$ which extends the locally free subsheaf given by $D(x), x \in U$.

Let $\phi: \PP^1 \to X$ be a maximally free curve. If it is not very free, then we have
\[
\phi^* {\Omega_X} \cong \OO(-a_1)\oplus \ldots \oplus \OO(-a_r) \oplus \OO \oplus\ldots \oplus \OO, r<n, a_1 \geq a_2 \geq \ldots a_r >0,
\]
\[
\phi^* \mathcal{D} \cong \OO\oplus \ldots \oplus \OO,
\]
(c.f. the paragraph after Corollary 3.2, loc. cit.)

Thus $\mathcal{D}$ is locally free along a general maximally free curve. We have an inclusion
\[
0 \to (\Lambda^{n-r}\mathcal{D})^{**} \to \Omega^{n-r}_X.
\]
Since $X$ is smooth, the reflexive rank $1$ sheaf $(\Lambda^{n-r}\mathcal{Q})^{**}$ is locally free. Furthermore, $X$ has Picard number one and $\deg \phi^*(\Lambda^{n-r}\mathcal{Q})^{**}=0$, thus it is trivial. So we have a section of $\Omega^{n-r}_X$.
\end{proof}

\begin{rem}
Koll\'ar (Exercise 5.19, Chap. V, \cite{Kollar96}) gives examples of separably uniruled, not separably rationally connected singular Fano varieties of Picard number one. The singularities are ordinary double points. Koll\'ar proved that they are not separably rationally connected by writing down unexpected sections of $\Omega_X^{n-1}$, where $n$ is the dimension. We do not know smooth examples.
\end{rem}

Before we state some applications of Theorem \ref{thm:PicOne}, we recall some definitions.
For every positive integer $N$, an $N$-\textbf{decomposition of the
    diagonal} is a finite sequence $(e_i:D_i\to X,Z_i)_i$ of pairs
  of a proper morphism $e_i$ with $D_i$ integral of dimension $n-1$
  and an $n$-cycle $Z_i\in \text{CH}_n(X\times_{\SP K} D_i)$ such
  that the $n$-cycle
  $$
  N\cdot[\Delta_X] - \sum_i (\text{Id}\times e_i)_*Z_i\in
  \text{CH}_n(X\times_{\SP K} X),
  $$
  is in the image of the flat pullback map,
  $$
  \text{pr}_1^*:\text{CH}_0(X) \to \text{CH}_n(X\times_{\SP K} X).
  $$
  Following \cite{ChatzLevine}, the \textbf{torsion order}
  $\text{Tor}(K,X)$ is the greatest common divisor of all integers
  $N$ such that there exists an $N$-decomposition of the diagonal.
  The torsion order equals $0$ if there is no such decomposition.
  
 \begin{cor}\label{cor:SRCnumeric}
Let $X$ be a smooth projective Fano variety, whose Picard group is isomorphic to $\ZZ$, defined over an algebraically closed field $\kappa=\bar{\kappa}$ of characteristic $p>0$. Assume that $p$ is prime to the torsion order $\text{Tor}(\bar{\kappa}, X)$ and the uniruling index $u_1(\bar{\kappa}, X)$. Then $X$ is separably rationally connected.
\end{cor} 

\begin{proof}
First of all, $X$ is separably uniruled since $p$ is prime to $u_1(\bar{\kappa}, X)$ (\cite[Corollary 3.3]{WASTZ2018}). 
Since $\text{Tor}(\bar{\kappa}, X)$ is prime to $p$, we have $H^0(X, \Omega_X^i)=0$, by the argument of Totaro \cite[Proof of Lemma 2.2]{TotaroRationality}.
Thus this follows from Theorem \ref{thm:PicOne}.
\end{proof}

\begin{cor}\label{cor:specializationPicOne}
Let $(R, \mfm, \kappa, K)$ be a DVR and  $\mcX \to \SP R$ be a smooth projective family. Assume that the Picard group of the geometric generic fiber $\mcX_{\bar{K}}$ is isomorphic to $\ZZ$, and that the characteristic of $\kappa$ is prime to the torsion order $\text{Tor}(\bar{K}, \mcX_{\bar{K}})$and the uniruling index $u_1(\bar{K}, \mcX_{\bar{K}})$of the geometric generic fiber $\mcX_{\bar{K}}$, then the central fiber is also separably rationally connected.
\end{cor}

\begin{proof}
The torsion order $\text{Tor}(\bar{\kappa}, \mcX_{\bar{\kappa}})$ (resp. the uniruling index $u_1(\bar{\kappa}, \mcX_{\bar{\kappa}})$) of the geometric closed fiber is a divisor of the torsion order $\text{Tor}(\bar{K}, \mcX_{\bar{K}})$ (resp. the uniruling index $u_1(\bar{K}, \mcX_{\bar{K}})$) of the geometric generic fiber. For the statement about torsion order, see, for example, \cite[Proposition 3.1, 3.2]{ChatzLevine}. For the statement about uniruling index, see \cite[Corollary IV. 1.7.4]{Kollar96}.
Thus by Corollary \ref{cor:SRCnumeric}, we only need to show that the Picard group of the geometric closed fiber is isomorphic to $\ZZ$.

We claim that for any (Weil) divisor $D$ in the geometric closed fiber $\mcX_{\kappa}$, $\text{Tor}(\bar{K}, \mcX_{\bar{K}}) \cdot D$ lies in the image of the specialization map $Sp: Pic(\mcX_{\bar{K}})\to Pic(\mcX_{\bar{\kappa}})$.
To see this, simply note that up to making extentions of $R$,  there are prime effective divisors $D_i \subset \mcX$, a section $s: \SP R \to \mcX$, such that $\text{Tor}(\bar{K}, \mcX_{\bar{K}}) [\Delta_{\mcX/R}] $ is rationally equivalent $\text{Tor}(\bar{K}, \mcX_{\bar{K}}) \mcX \times_R s(\SP R)+Z$, where $Z$ is a cycle supported in $\cup_i D_i \times_R \mcX$. Then the correspondence $\text{Tor}(\bar{K}, \mcX_{\bar{K}}) \Delta_{\mcX_{\bar{\kappa}}}$ acts on the Chow group of codimension $1$ cycles of $\mcX_{\bar{\kappa}}$ and $\text{Tor}(\bar{K}, \mcX_{\bar{K}})$ times any such cycle is a linear combination of ${D_i}|_{\mcX_{\bar{\kappa}}}$, which of course lies in the image of the specialization map.
On the other hand, we know the cokernal of $Sp$ is a finite $p$-group by \cite[Theorem 1.8]{GouJav}.
Thus the specialization is surjective and the Picard group of the geometric closed fiber is $\ZZ$.
\end{proof}

\begin{thm}\label{thm:WAPicOne}
For every smooth, affine, connected curve $B$ over a field
  $\fF=\ol{\fF}$, for every $B$-flat proper family
  $\mcX_B$.
  Assume that the geometric generic fiber $\mcX_{\fF(B)}$ is Fano and $Pic(X)\cong \ZZ$, and that $p$ is prime to the uniruling index and torsion order of the geometric generic fiber. Then weak approximation holds at every place
  of (strong) potentially good reduction.
\end{thm}

The proof is essentially the same as the proof of \cite[Corollary 1.18]{WASTZ2018}, except that now because of our stronger results about specializations of separable rational connectedness, we do not need the index one assumption anymore.

Note that for a complete intersection $X$ of type $(d_1, \ldots, d_c)$ such that $n>d_1+\ldots d_c$, there is a uniruling family of lines, the degree of the evaluation morphism of which is $\Pi_i (d_i !)$. Thus the uniruling index of $X$ is a divisor of $\Pi_i (d_i!)$. By \cite[Proposition 4.2]{ChatzLevine}, the torsion order is also a divisor of $\Pi_j (d_j !)$. So we obtain the following.

\begin{cor}\label{cor:SRCindex2}
Let $(R, \mfm, \kappa, K)$ be a DVR and  $\mcX \to \SP R$ be a smooth projective family. Assume that the generic fiber $\mcX_K$ is a smooth complete intersection in $\PP^n$ of type $(d_1, \ldots, d_c)$ such that $n>d_1+\ldots +d_c$. Furthermore, assume that the characteristic of the residue field $\kappa$ is greater than $\max(d_1, \ldots, d_c)$. Then the central fiber $\mcX_\kappa$ is separably rationally connected.
\end{cor}

Finally, we have the following version of weak approximation (cf.  \cite[Theorem 1.1]{WASTZ2018}).
\begin{thm}\label{thm:WACI}
For every smooth, affine, connected curve $B$ over a field
  $\fF=\ol{\fF}$, for every $B$-flat complete intersection
  $\mcX_B\subset B\times_{\SP \fF} \PP^n_{\fF}$ of $c$ hypersurfaces of
  degrees $(d_1,\dots,d_c)$, weak approximation holds at every place
  of (strong) potentially good reduction provided that the Fano index is $\geq 2$, and
  $p:=\text{char}(\fF)$  satisfies $p>\max(d_1,\dots,d_c)$.
\end{thm}

Here we replace the somewhat unnatural assumption ``(strong) potentially good reduction as complete intersections" in  \cite[Theorem 1.1]{WASTZ2018} with the usual (strong) potential good reduction assumption. This is possible since we know that for a place of (strong) potentially good reduction, the central closed fiber is still separably rationally connected, thanks to Corollary \ref{cor:SRCindex2}. Other than that, the proof goes exactly the same as that of \cite[Theorem 1.1]{WASTZ2018}. We refer the reader to \cite{WASTZ2018} for details of the proof.

\begin{rem}
It is an interesting question to study (smooth) deformations of complete intersections. A construction of Mori gives a deformation of hypersurface of degree $mn$ to a degree $m$ cyclic cover of a hypersurface of degree $n$. In particular,  this means that Theorem \ref{thm:WACI} is indeed an improvement of \cite[Theorem 1.1]{WASTZ2018}.
\end{rem}

We finish with a brief discussion of the higher Picard number case.

It is very easy to show that on a smooth projective separably rationally connected variety $X$, the group of rational one cycles modulo numerical equivalence $N_1(X)_\QQ$ is generated by (very) free rational curves. Thus if this group is not generated by free rational curves, then the variety $X$ is not separably rationally connected. The same argument as in the proof of Theorem \ref{thm:PicOne} gives the following.
\begin{thm}\label{thm:generalcase}
Let $X$ be a smooth projective variety over an algebraically closed field of positive characteristic $p$. Assume that $X$ is separably uniruled and the group of rational one cycles modulo numerical equivalence $N_1(X)_\QQ$ is generated by free rational curves. Then $X$ is separably rationally connected if and only if $H^0(X, \Omega^i_X)=0, i=1, \ldots, \dim X$.
\end{thm}

\begin{rem} \label{rem:projhomog}
  A stronger hypothesis is that $N_1(X)_\QQ$ is spanned by
  \emph{primitive} classes of free rational curves, $[R_i]$,
  $i=1,\dots,r$. Each $\mathbb{Z}_{\geq 0}\cdot[R_i]$ is the relative 
  Mori cone of a fiber-type contraction, $\phi_i:X\to X_i$.
  This holds for $X$ homogeneous: $X=G/P$.

For $X$ satisfying the stronger hypothesis, for a smooth complete
intersection $Y=Y_1\cap \dots \cap Y_c$ of $c$ ample hypersurfaces
$Y_i$ in $X$ with $\text{dim}(Y)\geq 3$, if each of the following
integers is positive,
$$
\langle c_1(T_X),[R_i]\rangle - 2 - \sum_{j=1}^c \langle [Y_j],[R_i] \rangle,
$$
and if also the characteristic $p$ is prime to each of the following
integers,
$$
f_{[R_i]}(Y,X) := \prod_{j=1}^c (\langle [Y_j],[R_i] \rangle)!,
$$
i.e., if also $p> \text{max}_{i,j}(\langle [Y_j],[R_i] \rangle)$, then
each $[R_i]$ is the pushforward from $Y$ of the class of a free
rational curve, and the stronger hypothesis also holds for $Y$.

It is also possible to bound the $2$-uniruling index of the generic
fiber of $\phi_i|_Y:Y\to \phi_i(Y)$ by chains of $R_i$-curves in terms
of invariants of chains of $R_i$-curves in the generic fiber of
$\phi_i$ and in terms of the integers $\langle [Y_j],[R_i]\rangle$.
If also $p$ is prime to this $2$-uniruling index, then $Y$ is
separably rationally connected.  Thus, there is an explicit version of
Corollary \ref{cor:SRCindex2} for complete intersections in projective
homogeneous spaces.
\end{rem}

\textbf{Acknowledgment:} Z.T. is partially supported by the program``Recruitment of global experts", and NSFC grants No. 11871155,
No. 11831013, No.11890662.
\bibliographystyle{alpha}
\bibliography{MyBib}

\end{document}